\documentclass[oneside,11pt]{amsart}

\usepackage{amssymb, mathrsfs, amscd}
\usepackage[all]{xy}

\usepackage{graphics}
\usepackage{graphicx}
\usepackage{enumerate}

\newtheorem*{thma}{Theorem A}
\newtheorem*{thmb}{Theorem B}

\newtheorem{theorem}{Theorem}[section]
\newtheorem{definition}[theorem]{Definition}
\newtheorem{lemma}[theorem]{Lemma}

\newtheorem{proposition}[theorem]{Proposition}
\newtheorem{corollary}[theorem]{Corollary}

\newtheorem{notation}[theorem]{Notation}

\author{Dieter Degrijse}
\address{Department of Mathematics, K.U.Leuven, Kortrijk, Belgium}%
\email{Dieter.Degrijse@kuleuven-kortrijk.be}%

\author{Nansen Petrosyan}
\address{Department of Mathematics, K.U.Leuven, Kortrijk, Belgium}%
\email{Nansen.Petrosyan@kuleuven-kortrijk.be}%
\thanks{Both authors were supported by the Research Fund K.U.Leuven.}
\thanks{The second author was also supported by the FWO-Flanders Research Fellowship.}

\title[]{Commensurators and classifying spaces  with virtually cyclic stabilizers}
\date{\today}
\newcommand{\mF}{\mathfrak {F}}
\newcommand{\Z}{\mathbb Z}
\newcommand{\G}{\Gamma}
\newcommand{\orb}{\mathcal{O}_{\mF}G}
\newcommand{\orbmod}{\mbox{Mod-}\mathcal{O}_{\mF}G}

\bigskip
\begin{document}
\maketitle
\begin{abstract} By examining commensurators of virtually cyclic groups, we show that for each natural number $n$, any locally finite-by-virtually cyclic group of cardinality $\aleph_n$ admits a finite dimensional classifying space with virtually cyclic stabilizers of dimension  $n+3$. As a corollary, we prove that every elementary amenable group of finite Hirsch length and cardinality $\aleph_n$ admits a finite dimensional classifying space with virtually cyclic stabilizers.
\end{abstract}
\section{Introduction}
A classifying space of a discrete group $G$ for a family of subgroups $\mathfrak{F}$ is a terminal object in the homotopy category of $G$-CW complexes with stabilizers in $\mathfrak{F}$ (see \cite{tom}). Such a space is also called a model for $E_{\mathfrak{F}}G$.  Even though a model for $E_{\mathfrak{F}}G$ always exists,  in general, it need not be finite dimensional. When $\mathfrak{F}$ is the family of virtually cyclic subgroups of $G$, $E_{\mathfrak{F}}G$ is denoted by $\underline{\underline{E}}G$. Questions concerning finiteness properties of $\underline{\underline{E}}G$, such as whether for a given type of group $G$ there exists a finite dimensional model for $\underline{\underline{E}}G$, have been particularly  motivated by Farell-Jones Isomorphism conjecture (see \cite{FJ} and \cite{DL}). More recently, this area has gathered  interest on its own.

Finite dimensional models for $\underline{\underline{E}}G$ are known to exists for several interesting classes of groups: e.g. word-hyperbolic groups (Juan-Pineda, Leary, \cite{LearyPineda}), relative hyperbolic groups (Lafont, Ortiz, \cite{LafontOrtiz}), virtually polycyclic groups (L\"{u}ck,Weiermann, \cite{LuckWeiermann}) and $\mathrm{CAT}(0)$-groups (L\"{u}ck, \cite{Luck3}). \\
\indent In a recent preprint \cite{FluchNucinkis}, Fluch and Nucinkis ask whether an elementary amenable group $\Gamma$ of finite Hirsch length has a finite dimensional model for $\underline{\underline{E}}\Gamma$. They give a positive answer to this question in the case where $\Gamma$ has a bound on the order of its finite subgroups. From the structural results of Hillman and Linnel (see \cite{HillmannLinnel}) and  Wehrfritz  (see \cite{Wehrfritz}) it follows that every elementary amenable group of finite Hirsch length $\Gamma$ is locally finite-by-virtually solvable where the virtually solvable group has a further decomposition.   Using this structure, Fluch and Nucinkis reduce the problem to torsion-free nilpotent-by-torsion-free abelian groups, and prove that these groups admit finite dimensional classifying spaces with virtually cyclic stabilizers.

In order to solve the problem for arbitrary  elementary amenable groups of finite Hirsch length, one  needs to show the existence of finite dimensional models for classifying spaces with virtually cyclic stabilizers of locally finite-by-virtually cyclic groups.

Let us denote the minimal dimension of a model for $\underline{\underline{E}}G$  by $\underline{\underline{\mathrm{gd}}}(G)$. We prove
\begin{thma}{\label{main}} Let $\Gamma$ be a group  of cardinality $\aleph_n$, for some natural number $n$. Suppose $\Gamma$ is an extension of an infinite locally finite group by a virtually cyclic group.  Then
\[  n+1 \leq \underline{\underline{\mathrm{gd}}}(\Gamma) \leq n+3.\]
\end{thma}
The proof of this theorem relies on a push-out construction of L\"{u}ck and Weiermann and a careful analysis of the structure of commensurators of virtually cyclic subgroups inside torsion-by-infinite cyclic groups. \\
\indent By applying Theorem A together with the cited results on torsion-free nilpotent-by-torsion-free abelian groups and elementary amenable groups of finite Hirsch length, we obtain
\begin{thmb}{\label{elem}} Let $\Gamma$ be an elementary amenable group of finite Hirsch length $h$ and cardinality $\aleph_n$, for some natural number $n$. Then there exists a function $f:\mathbb N\to \mathbb N$ and a finite dimensional model for $\underline{\underline{E}}\Gamma$ of dimension at most $f(h)$.
\end{thmb}

\section{Preliminaries} \label{sec: prelim}
Let $G$ be a discrete group and let $\mathfrak{F}$ be a family of subgroups of $G$, i.e. a collection of subgroups of $G$ that is closed under conjugation and taking subgroups. A \emph{classifying space of $G$ for the family $\mathfrak{F}$} is a $G$-CW-complex $X$ defined by  the properties that
 $X^H=\emptyset$ when $H \notin \mathfrak{F}$ and $X^H$ is contractible for every $H \in \mathfrak{F}$.

 A classifying space of $G$ for the family $\mathfrak{F}$ is also called a \emph{model for $E_{\mathfrak{F}}G$}. It can be shown that a model $X$ always exists and that it satisfies the following universal property: if $Y$ is a $G$-CW-complex such that all its stabilizer groups are contained in $\mathfrak{F}$, then there exists a $G$-equivariant map $Y \rightarrow X$ that is unique up to $G$-homotopy (e.g. see \cite{Luck2}, \cite{tom}). Note that this universal property implies that $X$ is unique up to $G$-homotopy. Equivalently, one could also say that a classifying space of a discrete group $G$ for a family of subgroups $\mathfrak{F}$ is a terminal object in the homotopy category of $G$-CW complexes with stabilizers in $\mathfrak{F}$ (e.g. see \cite{LuckMeintrup}).

The smallest possible dimension of a model for $E_{\mathfrak{F}}G$  is called the \emph{geometric dimension of $G$ for the family $\mathfrak{F}$} and  denoted by $\mathrm{gd}_{\mathfrak{F}}(G)$. When a finite dimensional model does not exist, then $\mathrm{gd}_{\mathfrak{F}}(G)$ is said to be infinite.

If $K$ is a subgroup of $G$, we can consider the family
\[ \mathfrak{F} \cap K= \{ K \cap H \ | \ H \in \mathfrak{F}\} \]
of subgroups of $K$. By restricting the action to $K$, any model for $E_{\mathfrak{F}}G$ becomes a model  for $E_{\mathfrak{F}\cap K}K$. This implies that $\mathrm{gd}_{\mathfrak{F} \cap K}(K) \leq \mathrm{gd}_{\mathfrak{F}}(G)$. When $K$ is a finite index subgroup of $G$, there is  a coinduction construction of  L\"{u}ck that entails the following
\begin{theorem}[{L\"{u}ck, \cite[2.4]{Luck1}}] \label{th: finite index} Let $\mathfrak{F}$ be either the family of finite subgroups of $G$ or the family of virtually cyclic subgroups of $G$. If $K$ is a subgroup of $G$ with finite index $[G:K]$, then
\[    \mathrm{gd}_{\mathfrak{F}}(G)  \leq  [G:K]\mathrm{gd}_{\mathfrak{F} \cap K}(K).   \]
\end{theorem}
\indent A general scheme to construct a model for $E_{\mathfrak{F}}G$ is to start with a model for $E_{\mathfrak{H}}G$ where $\mathfrak{H}$ is a subfamily of $\mathfrak{F}$, and then try to adapt this model to obtain a model for $E_{\mathfrak{F}}G$. A nice example of such an approach is a construction of L\"{u}ck and Weiermann (see \cite[$\S 2$]{LuckWeiermann}), which we will use here.

Let $G$ be a discrete group and let $\mathfrak{F}$ and $\mathfrak{H}$ be families of subgroups of $G$ such that $\mathfrak{H} \subseteq \mathfrak{F}$ and such that there exists an equivalence relation $\sim$ on the set $\mathcal{S}=\mathfrak{F}\smallsetminus \mathfrak{H}$ that satisfies the following properties
\begin{itemize}
\item[-]  $\forall H,K \in \mathcal{S} : H \subseteq K \Rightarrow H \sim K$;
\item[-] $ \forall H,K \in \mathcal{S},\forall x \in G: H \sim K \Leftrightarrow H^x \sim K^x$.
\end{itemize}
An equivalence relation that satisfying these properties will be called a \emph{strong equivalence relation}.

Let $[H]$ be an equivalence class represented by $H \in \mathcal{S}$ and denote the set of equivalence classes by $[\mathcal{S}]$. $G$ acts on $[\mathcal{S}]$ via conjugation, and the stabilizer group of an equivalence class $[H]$ is
\[ \mathrm{N}_{G}[H]=\{x \in \Gamma \ | \ H^x \sim H \}. \]
Note that $\mathrm{N}_{G}[H]$ contains $H$ as a subgroup. Let  $\mathcal{I}$ be a complete set of representatives $[H]$ of the orbits of the conjugation action of $G$ on $[\mathcal{S}]$.
For each $[H] \in \mathcal{I}$,  define the  family of subgroups of $\mathrm{N}_{G}[H]$
\[ \mathfrak{F}[H]=\{ K \subseteq\mathrm{N}_{G}[H] \ | K \in \mathcal{S}, K \sim H\} \cup \Big(\mathrm{N}_{G}[H] \cap \mathfrak{H}\Big). \]
\begin{theorem}[{L\"{u}ck-Weiermann, \cite{LuckWeiermann}}] \label{cor: push out} Let $\mathfrak{H} \subseteq\mathfrak{F}$ be two families of subgroups of a group $G$ such that $\mathcal{S}=\mathfrak{F}\smallsetminus \mathfrak{H}$ is equipped with a strong equivalence relation. Denote the set of equivalence classes by $[\mathcal{S}]$ and let $\mathcal{I}$ be a complete set of representatives $[H]$ of the orbits of the conjugation action of $G$ on $[\mathcal{S}]$. If there exists a natural number $n$ such that for each $[H] \in \mathcal{I}$
\begin{itemize}
\item[-] $\mathrm{gd}_{\mathfrak{H}\cap \mathrm{N}_{G}[H]}(\mathrm{N}_{G}[H]) \leq n-1$;
\item[-] $\mathrm{gd}_{\mathfrak{F}[H]}(\mathrm{N}_{G}[H]) \leq n$,
\end{itemize}
and such that $\mathrm{gd}_{\mathfrak{H}}G \leq n$, then $\mathrm{gd}_{\mathfrak{F}}G \leq n$.
\end{theorem}
An important algebraic tool to study the equivariant cohomology and finiteness properties of classifying spaces for families of subgroups is Bredon cohomology.
Bredon cohomology was introduced by Bredon in \cite{Bredon} for finite groups and has been generalized to arbitrary groups by L\"{u}ck (see \cite{Luck}). Let us recall some basics of the theory.

Let $G$ be a discrete group and let $\mathfrak{F}$ be a family of subgroups of $G$. The \emph{orbit category} $\orb$ is a category defined by the objects that are the left cosets $G/H$ for all $H \in \mathfrak{F}$ and the morphisms are all $G$-equivariant maps between the objects.
An \emph{$\orb$-module} is a contravariant functor $M: \orb \rightarrow \mathbb{Z}\mbox{-mod}$. The \emph{category of $\orb$-modules} is denoted by $\orbmod$ and is defined by the objects that are all the $\orb$-modules and the morphisms are all the natural transformations between the objects.
It can be shown that $\orbmod$ contains enough projective and injective objects to construct projective and injective resolutions. Hence, one can construct bi-functors $\mathrm{Ext}^{n}_{\orb}(-,-)$ that have all the usual properties. The \emph{$n$-th Bredon cohomology of $G$} with coefficients $M \in \orbmod$ is by definition
\[ \mathrm{H}^n_{\mathfrak{F}}(G,M)= \mathrm{Ext}^{n}_{\orb}(\underline{\mathbb{Z}},M), \]
where $\underline{\mathbb{Z}}$ is the constant functor. There is also a notion of \emph{cohomological dimension of $G$ for the family $\mathfrak{F}$}, denoted by $\mathrm{cd}_{\mathfrak{F}}(G)$ and defined by
\[ \mathrm{cd}_{\mathfrak{F}}(G) = \sup\{ n \in \mathbb{N} \ | \ \exists M \in \orbmod :  \mathrm{H}^n_{\mathfrak{F}}(G,M)\neq 0 \}. \]
Since the augmented cellular chain complex of any model for $E_{\mF}G$ yields a projective resolution of $\underline{\mathbb{Z}}$ that can be used to compute $\mathrm{H}_{\mF}^{\ast}(G,-)$, it follows that $ \mathrm{cd}_{\mathfrak{F}}(G) \leq  \mathrm{gd}_{\mathfrak{F}}(G)$. In fact, L\"{u}ck and Meintrup show that an even stronger result holds.
\begin{theorem}[{L\"{u}ck-Meintrup, \cite[0.1]{LuckMeintrup}}] \label{th: luck meintrup} Let $G$ be a group and let $\mathfrak{F}$ be a family of subgroups, then
\[ \mathrm{cd}_{\mathfrak{F}}(G) \leq  \mathrm{gd}_{\mathfrak{F}}(G) \leq \max\{3, \mathrm{cd}_{\mathfrak{F}}(G) \}. \]
\end{theorem}
Hence, to show that there exist a finite dimensional model for $E_{\mathfrak{F}}G$, it suffices to show that the Bredon cohomological dimension of $G$ for the family $\mathfrak{F}$ is finite. \\
\indent We finish this section with some notational conventions.
\begin{notation} \rm
\begin{itemize} \item[]
\item[-] If $a_1,a_2,\ldots,a_n$ are elements of some group $G$, then $\prod_{i=1}^na_i$ denotes the product $a_1a_2\ldots a_n$ in that specific order.
\item[-] By $\mathbb{N}_1$, we denote the set of natural numbers without the number zero.
\item[-] If $t$ is an element of some group $G$, then $\langle t \rangle$ is the cyclic subgroup of $G$ generated by $t$. All infinite cyclic groups will be written multiplicatively with unit $1$. The unit element of a group that is not infinite cyclic will be denoted by $e$.
\item[-] If $G$ is a group and $C=\langle t \rangle$ is an infinite cyclic group generated by $t$, we say that a semi-direct product $G\rtimes C$ is determined by $\varphi \in \mathrm{Aut}(G)$ when
\[ (e,t)(g,1)(e,t^{-1})=(\varphi(g),1) \in G \rtimes C\]
for all $g \in G$ and denote it by $G\rtimes_{\varphi} C$.
\item[-] Let $G$ be a group and let $\mathfrak{F}$ be a family of subgroups. If $\mathfrak{F}$ is the family of finite (virtually cyclic) subgroups then $E_{\mathfrak{F}}G$, $\mathrm{gd}_{\mathfrak{F}}G$ and $\mathrm{cd}_{\mathfrak{F}}G$ be will denoted by $\underline{E}G$ ($\underline{\underline{E}}G$), $\underline{\mathrm{gd}}G$ ($\underline{\underline{\mathrm{gd}}}G$) and $\underline{\mathrm{cd}}G$ ($\underline{\underline{\mathrm{cd}}}G$), respectively.
\end{itemize}
\end{notation}

\section{The structure of commensurators} \label{sec: comm}
In this section, we shall describe the structure of commensurators of infinite virtually cyclic groups in torsion-by-$\mathbb{Z}$ groups. We start with the following elementary but for our purposes indispensable  property of infinite virtually cyclic groups.
\begin{lemma} \label{lemma: intersection of vc}Let $H$ be an infinite virtually cyclic group and let $L_1$ and $L_2$ be two infinite virtually cyclic subgroups of $H$. Then $L_1 \cap L_2$ is infinite.
\end{lemma}
\begin{proof} By definition $H$ has an infinite cyclic subgroup $C$ of finite index. Let $t$ be a generator for $C$ and let $\{x_1,\ldots,x_n\}$ be a set of coset representatives of $H/C$. Because $L_1$ is infinite virtually cyclic, it has an element $x$ of infinite order. It is clear that for each $p \in \mathbb{N}_1$, there exists an integer $j_p$ and an element $x_{i_p} \in \{x_1,\ldots,x_n\}$ such that $x^p=x_{i_p}t^{j_p}$. Because there are only finitely many possibilities for $i_p$, there exist $m, n \in \mathbb{N}_1$ such that $n>m$ and $i_n=i_m$. It follows that $x^nt^{-j_n}=x^mt^{-j_m}$, and hence $x^{n-m}=t^{j_n-j_m}$. This implies that $L_1$ has an infinite cyclic subgroup $C_1=\langle x^{n-m}\rangle$ contained in $C$. Similarly  $L_2$ has an infinite cyclic subgroup $C_2$ contained in $C$. But it is clear that the intersection of two infinite subgroups in an infinite cyclic group is always infinite, hence $C_1 \cap C_2$ is infinite and contained in $L_1 \cap L_2$. This completes the proof.
\end{proof}
The commensurator of a subgroup $H$ of a group $\Gamma$ is by definition the group
\[ \mathrm{Comm}_{\Gamma}[H]=\{x \in \Gamma \ | \ [H: H \cap H^{x}]<\infty \ \mbox{and} \ [H^x: H \cap H^{x}]< \infty \}. \]
However, when $H$ is an infinite virtually cyclic subgroup of $\Gamma$ then one can verify that
\[ \mathrm{Comm}_{\Gamma}[H]=\{x \in \Gamma \ | \ |H \cap H^x | = \infty \}. \]

Let $\Gamma$ be a torsion-by-$\mathbb{Z}$ group. Then $\Gamma$ has a torsion normal subgroup $G$ such that $\Gamma/G$ is infinite cyclic. Hence, we have a split short exact sequence
\[  1 \rightarrow G \rightarrow \Gamma  \xrightarrow{\pi} \langle t \rangle \rightarrow 1 \]
such that multiplication in $\Gamma$ is given by $(a,t^n)(b,t^m)=(a\varphi^n(b),t^{n+m})$, for some fixed $\varphi \in \mathrm{Aut}(G)$.
This splitting of $\Gamma$ into a semi-direct product will be referred to as the $t$-splitting of $\Gamma$.

Let us denote the set of infinite virtually cyclic subgroups of $\Gamma$ by $\mathcal{S}$. As in definition $2.2.$ of \cite{LuckWeiermann}, we say that two infinite virtually cyclic subgroups $H$ and $K$ of $\Gamma$ are equivalent, denoted $H \sim K$, if $|H\cap K|=\infty$. Using Lemma \ref{lemma: intersection of vc}, it is not difficult to check that this indeed defines an equivalence relation on $\mathcal{S}$. One can also easily verify that this equivalence relation is a strong equivalence relation.

Now, suppose $H$ is an infinite virtually cyclic subgroup of $\Gamma$. We are interested in the structure of $\mathrm{Comm}_{\Gamma}[H]$. The following simple lemma shows that  $H$ can be assumed to be infinite cyclic.
\begin{lemma}  \label{lemma: easycommlemma} If $H$ and $K$ are two equivalent infinite virtually cyclic subgroups of $\Gamma$, then their commensurators coincide. In particular, the commensurator of an infinite virtually cyclic subgroup $H$ of $\Gamma$ is the commensurator of any infinite cyclic subgroup of $H$.
\end{lemma}
\begin{proof}
Suppose $H$ and $K$ are two infinite virtually cyclic subgroups of $\Gamma$ such that $H \sim K$. If $x \in \mathrm{Comm}_{\Gamma}[H]$, then by definition $H \sim H^x$. Since $H \sim K$, we have $H^x \sim K^x$ and therefore $K \sim K^x$. This implies that $\mathrm{Comm}_{\Gamma}[H]= \mathrm{Comm}_{\Gamma}[K]$. The second statement follows from this by noting that any element of $\mathcal{S}$ is equivalent to any of its infinite cyclic subgroups.
\end{proof}
The next lemma describes the structure of commensurators of infinite cyclic groups inside torsion-by-$\Z$ groups.

\begin{lemma} \label{lemma: keycommlemma} If $H$ is an infinite cyclic subgroup of $\Gamma$, then $\mathrm{Comm}_{\Gamma}[H]$ contains an infinite cyclic subgroup $K=\langle(a,t^k)\rangle$ of $\Gamma$ such that $H \sim K$ and such that there is a commutative diagram with split
exact rows
\[ \xymatrix{
1 \ar[r] & G  \ar[r] & \Gamma  \ar[r]^{\pi} & \langle t \rangle  \ar[r] & 1 \\
1 \ar[r] & T \ar[u]  \ar[r] &  \mathrm{Comm}_{\Gamma}[H] \ar[u]  \ar[r]^{\pi} & \langle t^k \rangle \ar[u]    \ar[r] & 1
} \]
where the vertical maps are inclusions,
\[  T= \{ g \in G \ | \ \exists n \in \mathbb{N}_1: \varphi^{nk}(g)=\alpha_{n,k}^{-1} g \alpha_{n,k} \}, \]
and $\alpha_{n,k}=\prod_{i=0}^{n-1}\varphi^{ik}(a)$.
\end{lemma}
\begin{proof} Let us view $\mathrm{Comm}_{\Gamma}[H]$ as a subgroup of $\Gamma$ where $\Gamma$ is given by the $t$-splitting. Clearly, $\pi(\mathrm{Comm}_{\Gamma}[H])=\langle t^k \rangle$ for some $k \in \mathbb{N}_1$ and the kernel of $\pi_{|\mathrm{Comm}_{\Gamma}[H]}$ is a torsion group, which we denote by $T$. Now consider an element $(a,t^k)\in \mathrm{Comm}_{\Gamma}[H] \subseteq \Gamma$ and denote it by $s$. Then $\mathrm{Comm}_{\Gamma}[H]$ can be written as a semi-direct product $T \rtimes \langle s \rangle$ determined by some $\theta \in \mathrm{Aut}(T)$. The splitting of $\mathrm{Comm}_{\Gamma}[H]$ into this particular semi-direct product will be referred to as the $s$-splitting of $\mathrm{Comm}_{\Gamma}[H]$.

By considering the $s$-splitting of $\mathrm{Comm}_{\Gamma}[H]$ and recalling that $H$ is a subgroup of $\mathrm{Comm}_{\Gamma}[H]$, it follows that $H$ is generated by some $(h,s^m) \in \mathrm{Comm}_{\Gamma}[H]$. Since $(e,s) \in \mathrm{Comm}_{\Gamma}[H]$, we must have that $\langle {}^{(e,s)}(h,s^m)\rangle\sim \langle(h,s^m) \rangle$. This implies that there exists  $n \in \mathbb{N}_1$ such that \[ (e,s)(h,s^m)^n(e,s^{-1})=(h,s^m)^n.\]
Denoting $\beta_{n,m}=\prod_{i=0}^{n-1}\theta^{im}(h)$,  the equation above implies that $\theta(\beta_{n,m})=\beta_{n,m}$. From this it follows that $(\beta_{n,m},s^{mn})^k= (\beta_{n,m}^k,s^{mnk})$ for each $k \in\mathbb{N}_1$. Since $\beta_{n,m}$ has finite order, there exists  $r \in \mathbb{N}_1$ such that $(\beta_{n,m},s^{mn})^r=(e,s^{mnr})$. We conclude that
\[ \langle(h,s^m)\rangle \sim \langle(h,s^m)^{nr}\rangle \sim \langle(e,s)^{mnr}\rangle \sim \langle(e,s)\rangle, \]
hence  $H=\langle(h,s^m)\rangle \sim \langle(e,s)\rangle:=K$. As a subgroup of $\Gamma$, $K$ is generated by $(a,t^k)$. Thus, we have found a infinite cyclic subgroup $K=\langle(a,t^k)\rangle$ of $\Gamma$ such that $H \sim K $ and such that \[ \xymatrix{
1 \ar[r] & G  \ar[r] & \Gamma  \ar[r]^{\pi} & \langle t \rangle  \ar[r] & 1 \\
1 \ar[r] & T \ar[u] \ar[r] &  \mathrm{Comm}_{\Gamma}[H] \ar[u] \ar[r]^{\pi} & \langle t^k \rangle \ar[u]  \ar[r] & 1
} \] commutes and has split exact rows.  Also, by Lemma \ref{lemma: easycommlemma}, we have $\mathrm{Comm}_{\Gamma}[H]=\mathrm{Comm}_{\Gamma}[K]$. It remains to determine $T$.

 It is clear from the diagram that $T=G\cap\mathrm{Comm}_{\Gamma}[H]$. Since  $\mathrm{Comm}_{\Gamma}[H]=\mathrm{Comm}_{\Gamma}[K]$, it follows that $T=G\cap\mathrm{Comm}_{\Gamma}[\langle(a,t^k)\rangle]$. By considering the $t$-splitting of $\Gamma$ and recalling that $K=\langle(a,t^k)\rangle$ under this splitting, we find that given $g \in G$, then $g \in T$ if and only if $(g,1)(a,t^k)^n(g^{-1},1)=(a,t^k)^n$ for some $n \in \mathbb{N}_1$. Since $(a,t^k)^n=(\alpha_{n,k},t^{kn})$ for $\alpha_{n,k}=\prod_{i=0}^{n-1}\varphi^{ik}(a)$, it follows that given $g \in G$, then $g \in T$ if and only if $\varphi^{nk}(g)=\alpha_{n,k}^{-1} g \alpha_{n,k}$. So,
\[  T= \{ g \in G \ | \ \exists n \in \mathbb{N}_1: \varphi^{nk}(g)=\alpha_{n,k}^{-1} g \alpha_{n,k} \}. \]
\end{proof}
\begin{proposition} \label{prop: comm} Let $\Gamma=G\rtimes \mathbb{Z}$, where $G$ a torsion group and let $H$ be an infinite virtually cyclic subgroup of $\Gamma$. Then the following hold.
\begin{itemize}
\item[(a)] There exists a subgroup $T$ of $G$ such that $\mathrm{Comm}_{\Gamma}[H]\cong T\rtimes_{\theta} \mathbb{Z}$ and  such that for each $g \in T$ there is  $n \in \mathbb{N}_1$ for which $\theta^n(g)=g$.
\item[(b)] Every infinite virtually cyclic subgroup of $\mathrm{Comm}_{\Gamma}[H]$ is equivalent to $H$.
\end{itemize}
\end{proposition}
\begin{proof}
Let $C$ be an infinite cyclic subgroup of $H$. It follows from Lemma \ref{lemma: easycommlemma} that $\mathrm{Comm}_{\Gamma}[H]=\mathrm{Comm}_{\Gamma}[C]$. Moreover, any infinite virtually cyclic subgroup of $\mathrm{Comm}_{\Gamma}[H]$ that is equivalent to $H$ is also equivalent to $C$, and vice versa.

Let us view $\Gamma$ as a semi-direct product $G\rtimes_{\varphi}\langle t \rangle$, i.e. fix a $t$-splitting for $\Gamma$. It follows from Lemma \ref{lemma: keycommlemma} that $\mathrm{Comm}_{\Gamma}[H]$ contains an infinite cyclic subgroup $K=\langle(a,t^k)\rangle$ of $\Gamma$ such that $K \sim C \sim H$ and such that there is a commutative diagram with split
exact rows
\[ \xymatrix{
1 \ar[r] & G  \ar[r] & \Gamma  \ar[r]^{\pi} & \langle t \rangle  \ar[r] & 1 \\
1 \ar[r] & T  \ar[u]  \ar[r] &  \mathrm{Comm}_{\Gamma}[H]  \ar[u]  \ar[r]^{\pi} & \langle t^k \rangle  \ar[u]  \ar[r] & 1
} \]
where the vertical maps are inclusions and
\[  T= \{ g \in G \ | \ \exists n \in \mathbb{N}_1: \varphi^{nk}(g)=\alpha_{n,k}^{-1} g \alpha_{n,k} \}, \]
and $\alpha_{n,k}=\prod_{i=0}^{n-1}\varphi^{ik}(a)$.
Now, let us consider the group $\mathrm{Comm}_{\Gamma}[H]$ in the diagram above, and view it as a semi-direct product via the $s$-splitting determined by the element $s:=(a,t^k)$. It follows that $\mathrm{Comm}_{\Gamma}[H] \cong T \rtimes_{\theta} \langle s \rangle$, where  $\theta: T \rightarrow T : g \mapsto a\varphi^k(g)a^{-1}$. One easily verifies that \[  T= \{ g \in G \ | \ \exists n \in \mathbb{N}_1: \theta^n(g)=g \}. \]
This proves part (a) of the proposition.

Suppose $V$ is an infinite virtually cyclic subgroup of $ \mathrm{Comm}_{\Gamma}[H]$ and let $S$ be an infinite cyclic subgroup of $V$. We want to show that $V \sim H$. By Lemma \ref{lemma: easycommlemma} and the fact that $H \sim K$, this is equivalent to showing that $S \sim K$. Under the $s$-splitting of $\mathrm{Comm}_{\Gamma}[H]$, $S$ is generated by some element $(g,s^m)$ and $K$ is generated by $(e,s)$. By part (a),  there exists $n \in \mathbb{N}_1$ such that $\theta^n(g)=g$. Denote $b= \prod_{i=0}^{n-1}\theta^{im}(g)$ and let $p \in \mathbb{N}_1$ such that $b^p=e$ in $G$. It now follows that $(g,s^m)^{np}=( \prod_{i=0}^{np-1}\theta^{im}(g),s^{mnp})=(b^p,s^{mnp})=(e,s^{mnp})$. This implies that $S \sim K$ and finishes the proof of part (b).
\end{proof}
\section{Proofs of Theorems A and B}
Throughout this section let  $n$ be a fixed nonnegative integer. We begin by determining an upper bound for $\underline{\mathrm{gd}}(\Gamma)$.
\begin{lemma} \label{lemma: bound finite} Let $\Gamma$ be a locally finite-by-virtually cyclic group of  cardinality $\aleph_n$. Then $\underline{\mathrm{gd}}(\Gamma)\leq n+2$.
\end{lemma}
\begin{proof} We have a short exact sequence
\[ 1 \rightarrow G \rightarrow \Gamma \xrightarrow{\pi} V \rightarrow 1 \]
where $G$ is locally finite of cardinality $\aleph_n$ and $V$ is virtually cyclic.
Let us denote by $\mathfrak{F}$ the family of finite subgroups of $\Gamma$, and denote by $\mathfrak{V}$  the family of finite subgroups of $V$. Now, we define the family
\[ \mathfrak{H}=\{ K \subseteq \Gamma \ | \ K \subseteq \pi^{-1}(F) \ \mbox{for some} \ F \in \mathfrak{V}\} \]
of subgroups of $\Gamma$. Note that $\mathfrak{F} \subseteq \mathfrak{H}$. Let $X$ be a model for $E_{\mathfrak{V}}V$. By letting $\Gamma$ act on $X$ via $\pi$, it follows that $X$ is also a model for $E_{\mathfrak{H}}\Gamma$. This, together with Proposition $4$ of \cite{LearyPineda}, implies that $\mathrm{gd}_{\mathfrak{H}}(\Gamma) \leq \mathrm{gd}_{\mathfrak{V}}(V)\leq 1$.

Note that each group in $\mathfrak{H}$ is locally finite of cardinality at most $\aleph_n$. It therefore follows from Theorem $2.6$ of \cite{DicksKrophollerLearyThomas} that $\mathrm{gd}_{\mathfrak{F}\cap H}(H)=\underline{\mathrm{gd}}(H) \leq n+1$ for all $H \in \mathfrak{H}$. Proposition $5.1$(i) of \cite{LuckWeiermann} now shows that $\underline{\mathrm{gd}}(\Gamma)\leq n+2$.
\end{proof}
\begin{definition} \rm We say that a semi-direct product group $G \rtimes_{\varphi} \mathbb{Z}$  is  \emph{locally bounded} if for each $g \in G$, there exists some $n \in \mathbb{N}_1$ such that $\varphi^n(g)=g$.
\end{definition}

The following proposition determines $\underline{\underline{\mathrm{gd}}}(G \rtimes_{\varphi} \mathbb{Z})$, when $G \rtimes_{\varphi} \mathbb{Z}$ is locally bounded.
\begin{proposition} \label{prop: directed union} Let $G$ be an infinite locally finite group of cardinality $\aleph_n$ and assume that $G \rtimes_{\varphi} \mathbb{Z}$ is locally bounded. If a group $\Gamma$ contains $G \rtimes_{\varphi} \mathbb{Z}$ as a subgroup of finite index, then $\Gamma$ is a directed union of its virtually cyclic  subgroups and $\underline{\underline{\mathrm{gd}}}(\Gamma)= n+1$.
\end{proposition}
\begin{proof} We will first show that  $G \rtimes_{\varphi} \mathbb{Z}$ is a directed union of  virtually cyclic subgroups.

Let $\mathcal{P}_{fin}(G)$ be the set of finite subsets of $G$ and consider an arbitrary subset $S=\{x_1,\ldots,x_s\} \in \mathcal{P}_{fin}(G)$. Since $G \rtimes_{\varphi} \mathbb{Z}$ is locally bounded,  for each $i \in \{1,\ldots,s\}$, there exists an integer $n_i \in \mathbb{N}_1$ such that $\varphi^{n_i}(x_i)=x_i$. Let $F_S$ be the subgroup of $G$ generated by the set of elements \[ \{\varphi^{j_i}(x_i) \ | \ i\in \{1,\ldots,s\} \ \mbox{and} \ j_i\in \{1,\ldots,n_{i}\}\}.\] Since $G$ is locally finite, $F_S$ is a finite group for each $S \in \mathcal{P}_{fin}(G)$. One also easily verifies that $\varphi(F_S)=F_S$ for each $S  \in \mathcal{P}_{fin}(G)$, and that $F_S \subseteq F_{S'}$ if $S \subseteq S'$.
It follows that we can use $\varphi$ to construct the semi-direct products $H_S:=F_S\rtimes \mathbb{Z}$ for all $S \in \mathcal{P}_{fin}(G)$, such that
\[ \xymatrix{
1 \ar[r] & F_S \ar[d] \ar[r] & H_S \ar[d] \ar[r] & \mathbb{Z} \ar[d]^{\mathrm{Id}} \ar[r] & 0 \\
1 \ar[r] & G  \ar[r] & G\rtimes_{\varphi} \mathbb{Z}  \ar[r] & \mathbb{Z}  \ar[r] & 0,
} \]
commutes and the vertical maps are inclusions. The set $\{H_S \ | \ S \in \mathcal{P}_{fin}(G)\}$ ordered by inclusion is a directed set of virtually cyclic subgroup of $G \rtimes_{\varphi} \mathbb{Z}$,
because $H_{S_1}$ and $H_{S_2}$ are both contained in $H_{S_1 \cup S_2}$. Since every element of $G\rtimes_{\varphi} \mathbb{Z}$ is contained in some $H_S$, it follows that $G \rtimes_{\varphi} \mathbb{Z}$ is the directed union of the virtually cyclic groups $H_S$ for all $S \in \mathcal{P}_{fin}(G)$.
One can therefore conclude that the family  of finitely generated subgroups of $G \rtimes_{\varphi} \mathbb{Z}$ coincides with the family of virtually cyclic subgroups.

Now, by Theorem $5.31$ of \cite{LuckWeiermann}, we have that $\underline{\underline{\mathrm{gd}}}(G \rtimes_{\varphi} \mathbb{Z})\leq n+1$. Since $\underline{\mathrm{gd}}(G)= n+1$ (see \cite[2.6]{DicksKrophollerLearyThomas}) and $\underline{\mathrm{gd}}(G)= \underline{\underline{\mathrm{gd}}}(G)\leq \underline{\underline{\mathrm{gd}}}(G \rtimes_{\varphi} \mathbb{Z})$, it follows that $\underline{\underline{\mathrm{gd}}}(G \rtimes_{\varphi} \mathbb{Z})= n+1$.

Finally, let us suppose that a group $\Gamma$ contains $G \rtimes_{\varphi} \mathbb{Z}$ as a subgroup of finite index.
We will show that $\Gamma$ is locally virtually cyclic. This is equivalent with $\Gamma$ being a directed union of virtually cyclic subgroups and implies that $\underline{\underline{\mathrm{gd}}}(\Gamma)= n+1$, as before.

Let $K$ be a finitely generated subgroup of $\Gamma$.
It follows that $K_0=K \cap (G \rtimes_{\varphi} \mathbb{Z})$ is a finite index subgroup of $K$ and is therefore also finitely generated.
But then $K_0$ is virtually cyclic because $G \rtimes_{\varphi} \mathbb{Z}$ is locally virtually cyclic.
Since  $K$  is a finite extension of $K_0$ it is also virtually cyclic. This shows that $\Gamma$ is locally virtually cyclic.
\end{proof}
We are now ready to prove Theorem A.
\begin{proof}[Proof of Theorem A]
Recall that every virtually cyclic group is either finite, finite-by-infinite cyclic or finite-by-infinite dihedral. Hence, there exists a locally finite group $G$ of cardinality $\aleph_n$ such that $\Gamma$ is either $G$, $G$-by-infinite cyclic or $G$-by-infinite dihedral. If $\Gamma$ is locally finite, then  by  Theorem $2.6$ of \cite{DicksKrophollerLearyThomas}, $\underline{\underline{\mathrm{gd}}}(\Gamma)=\underline{\mathrm{gd}}(\Gamma)= n+1$.

Now, let us assume that $\Gamma$ is not locally finite. Let $\mathcal{S}$ be the set of infinite virtually cyclic subgroups of $\Gamma$ and equip $\mathcal{S}$ with the equivalence relation $\sim$ discussed in Section \ref{sec: comm}. Recall that this is a strong equivalence relation. Denote an equivalence class represented by $H \in \mathcal{S}$ by $[H]$, and denote the set of equivalence classes by $[\mathcal{S}]$. Note that $\Gamma$ acts on $[\mathcal{S}]$ via conjugation and the stabilizer subgroup of an equivalence class $[H]$ under this action is exactly the commensurator $\mathrm{Comm}_{\Gamma}[H]$ of any representative $H$, i.e. $\mathrm{N}_{\Gamma}[H]=\mathrm{Comm}_{\Gamma}[H]$.
Denote by $\mathcal{I}$ a complete set of representatives $[H]$ of the $\Gamma$-orbits of $[\mathcal{S}]$ and define for each $[H] \in \mathcal{I}$ the following family of subgroups of $\mathrm{Comm}_{\Gamma}[H]$:
\[ \mathfrak{F}[H]=\{ K \subseteq\mathrm{Comm}_{\Gamma}[H] \ | K \in \mathcal{S}, K \sim H\} \cup \{ K \subseteq\mathrm{Comm}_{\Gamma}[H] \ | \ |K| < \infty \}. \]

We claim that there exists a model for $E_{\mathfrak{F}[H]}(\mathrm{Comm}_{\Gamma}[H])$ of dimension at most $n+1$, for each $[H] \in \mathcal{I}$. To prove this, first let us assume that $\Gamma$ is locally finite-by-infinite cyclic. It follows from  Proposition \ref{prop: comm}(a) that for each $[H] \in \mathcal{I}$, $\mathrm{Comm}_{\Gamma}[H]$ is isomorphic to some semi-direct product $T\rtimes \mathbb{Z}$ that is locally bounded, with $T$ locally finite of cardinality at most $\aleph_n$. It follows from  Proposition \ref{prop: comm}(b) that the family $\mathfrak{F}[H]$ coincides with the family of virtually cyclic subgroups of $\mathrm{Comm}_{\Gamma}[H]$. Therefore, Proposition \ref{prop: directed union} implies that there exists a model for $E_{\mathfrak{F}[H]}(\mathrm{Comm}_{\Gamma}[H])$ of dimension $n+1$, for each $[H] \in \mathcal{I}$.

If $\Gamma$ is not locally finite-by-infinite cyclic, then $\Gamma$ must have an index $2$ subgroup $\Gamma_0$ that is locally finite-by-infinite cyclic. Note that in this case, each $[H] \in \mathcal{I}$ can be represented by an infinite virtually cyclic subgroup $H$  contained in $\Gamma_0$, and that $\mathrm{Comm}_{\Gamma_0}[H]=\mathrm{Comm}_{\Gamma}[H]\cap \Gamma_0$. This implies that $\mathrm{Comm}_{\Gamma_0}[H]$ is a subgroup of $\mathrm{Comm}_{\Gamma}[H]$ of index at most $2$. We leave it to the reader to check, using   Proposition \ref{prop: comm}(b) and the observations above, that the family $\mathfrak{F}[H]$ coincides with the family of virtually cyclic subgroups of $\mathrm{Comm}_{\Gamma}[H]$. It now follows from  Proposition \ref{prop: comm}(a) and Proposition \ref{prop: directed union} that there exists a model for $E_{\mathfrak{F}[H]}(\mathrm{Comm}_{\Gamma}[H])$ of dimension at most $n+1$, for each $[H] \in \mathcal{I}$. This proves the claim.

By Lemma \ref{lemma: bound finite}, there is model for $\underline{E}\Gamma$ of dimension  $n+2$. This is also a model for $\underline{E}\mathrm{Comm}_{\Gamma}[H]$ for each $[H] \in \mathcal{I}$.
Applying Theorem \ref{cor: push out}, yields a model for $\underline{\underline{E}}G$ of dimension  $n+3$.
Since  $n+1=\underline{\mathrm{gd}}(G)\leq \underline{\underline{\mathrm{gd}}}(\Gamma)$, we obtain
\[  n+1 \leq \underline{\underline{\mathrm{gd}}}(\Gamma) \leq n+3.\]\end{proof}

\begin{corollary} \label{cor: ext of loc fin} Suppose we have a group extension
\[ 1 \rightarrow G \rightarrow \Gamma \xrightarrow{\pi} Q \rightarrow 1 \]
such that $G$ is locally finite of cardinality $\aleph_n$. Then \[ \underline{\underline{\mathrm{gd}}}(\Gamma) \leq n + 3 + \underline{\underline{\mathrm{cd}}}(Q) .\]
\end{corollary}
\begin{proof}
The statement follows by combining Corollary 5.2 of \cite{Martinez}, Theorem A and Theorem \ref{th: luck meintrup}.
\end{proof}
We can now prove Theorem B.

\begin{proof}[Proof of Theorem B] It follows from statement (g) of \cite{Wehrfritz} (see also \cite{HillmannLinnel}) that there exist characteristic subgroups $\Lambda(\G)\subseteq N\subseteq M$ of $\G$ and a function $j: \mathbb{N} \rightarrow \mathbb{N}$ such that
\begin{itemize}
\item[-] $\Lambda(\G)$ is the unique maximal normal locally finite subgroup of $\G$, i.e. the locally finite radical;
\item[-] $N/\Lambda(\G)$ is torsion-free nilpotent;
\item[-] $M/N$ is free abelian of finite rank;
\item[-] $[\G:M] \leq j(h)$.
\end{itemize}

By Theorem \ref{th: finite index}, it suffices to prove that $\underline{\underline{\mathrm{gd}}}(M)$ is bounded above by an integer valued function of $h$. Also by the above, we have a short exact sequence
\[ 1 \rightarrow N/\Lambda(\G) \rightarrow M/\Lambda(\G) \rightarrow M/N \rightarrow 1. \]
Since the nilpotency class of any torsion-free nilpotent group is at most its Hirsch length, a close inspection of the proof of Theorem $6.2$ of \cite{FluchNucinkis} immediately implies that $\underline{\underline{\mathrm{cd}}}(M/\Lambda(\G))\leq g(h)$ for some function $g: \mathbb{N}_0\to \mathbb N_0$.

Now, let us consider the short exact sequence
\[ 1 \rightarrow \Lambda(\G) \rightarrow M \rightarrow  M/\Lambda(\G) \rightarrow 1. \]
Since $\Lambda(\G)$ is a locally finite group of cardinality at most $\aleph_n$, it follows from Corollary \ref{cor: ext of loc fin} that $\underline{\underline{\mathrm{gd}}}(M) \leq n+3+ g(h)$. This finishes the proof.
\end{proof}

$${\bf Acknowledgment}$$

The authors wish to thank  Brita Nucinkis for introducing them to the problems discussed in this paper.

\end{document}